\documentclass[12pt]{amsart}
\usepackage{amssymb}
\usepackage{amsfonts}
\usepackage{amssymb,latexsym}
\usepackage{enumerate}
\usepackage{mathrsfs}
 \usepackage{color}\makeatletter
\@namedef{subjclassname@2010}{%
  \textup{2010} Mathematics Subject Classification}
\makeatother

  [2002/01/19 v2.2g %
    AMS font definitions%
  ]
\DeclareFontFamily{U}{euf}{}
\DeclareFontShape{U}{euf}{m}{n}{%
  <5><6><7><8><9>gen*eufm%
  <10><10.95><12><14.4><17.28><20.74><24.88>eufm10%
  }{}
\DeclareFontShape{U}{euf}{b}{n}{%
  <5><6><7><8><9>gen*eufb%
  <10><10.95><12><14.4><17.28><20.74><24.88>eufb10%
  }{}

  [2002/01/19 v2.2g %
    AMS font definitions%
  ]
\DeclareFontFamily{U}{msb}{}
\DeclareFontShape{U}{msb}{m}{n}{%
  <5><6><7><8><9>gen*msbm%
  <10><10.95><12><14.4><17.28><20.74><24.88>msbm10%
  }{}

  [2002/01/19 v2.2g %
    AMS font definitions%
  ]
\DeclareFontFamily{U}{msa}{}
\DeclareFontShape{U}{msa}{m}{n}{%
  <5><6><7><8><9>gen*msam%
  <10><10.95><12><14.4><17.28><20.74><24.88>msam10%
  }{}

\newtheorem{theorem}{Theorem}[section]
\newtheorem{lemma}[theorem]{Lemma}
\newtheorem{proposition}[theorem]{Proposition}
\newtheorem{corollary}[theorem]{Corollary}

\theoremstyle{definition}

\theoremstyle{remark}

\numberwithin{equation}{section} 

\def\Re{{\rm Re}}

\begin{document}

\title[]
{Jackson's integral of multiple Hurwitz-Lerch zeta functions and multiple gamma functions}

\begin{abstract}
Using the Jackson integral, we obtain the $q$-integral analogue  of  Raabe's  1843  formula for Barnes multiple Hurwitz-Lerch zeta functions and Barnes and Vardi's multiple gamma functions.

Our results generalize $q$-integral analogue of the  Raabe type formulas for the Hurwitz zeta functions and log gamma functions in [N. Kurokawa, K. Mimachi, and M. Wakayama,
Jackson's integral of the Hurwitz zeta function, Rend. Circ. Mat. Palermo (2) 56 (2007), no. 1, 43--56]. During the proof we also obtain a new formula on the relationship between the higher and lower orders  Hurwitz zeta functions.
\end{abstract}

\author{Su Hu}
\address{Department of Mathematics, South China University of Technology, Guangzhou, Guangdong 510640, China}
\email{mahusu@scut.edu.cn}

\author{Daeyeoul Kim}
\address{Department of Mathematics and Institute of Pure and Applied Mathematics, Chon-buk National University, 567 Baekje-daero, deokjin-gu, Jeonju-si,
Jeollabuk-do 54896, Republic of Korea}
\email{kdaeyeoul@jbnu.ac.kr}

\author{Min-Soo Kim}
\address{Division of Mathematics, Science, and Computers, Kyungnam University, 7(Woryeong-dong) kyungnamdaehak-ro, Masanhappo-gu, Changwon-si,
Gyeongsangnam-do 51767, Republic of Korea}
\email{mskim@kyungnam.ac.kr}



\maketitle

\vspace*{1ex}

\def\C{\mathbb C_p}
\def\BZ{\mathbb Z}
\def\Z{\mathbb Z_p}
\def\Q{\mathbb Q_p}
\def\nu{{\rm ord}}

\section{Introduction}\label{intro}
Raabe's 1843 formula for Gamma function is as follows:
\begin{equation}\label{raab1}
\int_{0}^{1}\log\Big(\frac{\Gamma(a+t)}{\sqrt{2\pi}}\Big)dt=a\log a-a, \quad a\geq 0
\end{equation}
(see \cite[p. 89]{Nielsen}; see also \cite[p. 29, Eq. (40) and Eq. (41)]{SC}).

The Hurwitz zeta function $\zeta(s,a)$ is
defined by
\begin{equation}\label{zeta1}
\zeta(s,a)=\sum_{n=0}^\infty\frac{1}{(n+a)^s}
\end{equation}
for $\Re(s)>1$ and $a\in\mathbb C\setminus\mathbb Z_0^-.$
Here $\mathbb Z_0^-=\{0,-1,-2,\ldots\}$ and $\mathbb C$ is the set of complex numbers.
It also satisfies the Raabe type formula:
\begin{equation}\label{zeta2}
\int_{0}^{1}\zeta(s,a+t)dt=\frac{a^{1-s}}{s-1}.
\end{equation}

The gamma function appears in the derivative of the Hurwitz zeta function as indicated by the following Lerch's formula \cite{Lerch} (see also \cite[p. 159, Eq. (17)]{SC}):
\begin{equation*}
\frac{\partial}{\partial s} \zeta(s,a)\Big |_{s=0}  =\log \frac{\Gamma(a)}{\sqrt{2\pi}}.
\end{equation*}

The Jackson integral is, when $q>1,$ defined as
\begin{equation}\label{j-int}
\int_0^1 f(a) d_qa=(q-1)\sum_{n=1}^\infty f(q^{-n})q^{-n}
\end{equation}(see \cite{Jack}).
Using (\ref{j-int}), Kurokawa, Mimachi and Wakayama  proved  $q$-integral analogue of the  Raabe type formulas for the Hurwitz zeta functions $\zeta(s,a)$ and log Gamma functions $\log \Gamma(a)$, respectively. (See \cite[Theorems 2 and 3]{KMW}).

The multiple (or, simply, $k$-tuple) Hurwitz-Lerch zeta function $\Phi_k(z,s,a)$ is defined by
\begin{equation}\label{m-g-zeta}
\Phi_k(z,s,a)=\sum_{m_1,\ldots,m_k=0}^\infty\frac{z^{m_1+\dots+m_k}}{(m_1+\dots+m_k+a)^{s}}
\end{equation}
for $a\in\mathbb C\setminus\mathbb Z_0^-, s\in\mathbb C$ when $|z|<1$ and $\Re(s)>k$ when $|z|=1$ (see \cite[(6)]{CJS}).
It is analytically continued to $s\in\mathbb C$ as a meromorphic function by the usual method (see \cite{Be}) and
holomorphic at $s\in\mathbb C\setminus\{1,\ldots,k\}$ when $|z|=1.$
In its special case when $z = 1,$ the definition (\ref{m-g-zeta}) yields the familiar multiple (or, simply, $k$-tuple)
Hurwitz zeta function $\zeta_k(s,a)=\Phi_k(1,s,a)$:
\begin{equation}\label{m-zeta_z=1}
\zeta_k(s,a)=\sum_{m_1,\ldots,m_k=0}^\infty\frac{1}{(m_1+\dots+m_k+a)^{s}},
\end{equation}
where $\Re(s)>k$ and $a\in\mathbb C\setminus\mathbb Z_0^-.$

Letting $k=1$ in (\ref{m-g-zeta}),  we obtain the Hurwitz-Lerch zeta function $$\Phi_1(z,s,a)=\Phi(z,s,a).$$
Furthermore, when $k=1$ and $z=1,$  (\ref{m-g-zeta}) also  yields the Hurwitz  zeta function
$\Phi_1(1,s,a)=\zeta(s,a).$
In addition, letting $k=1$ and $a=1$ in (\ref{m-zeta_z=1}), the Riemann zeta function $\zeta_1(s,1)=\zeta(s)$ is obtained (for details, see \cite{Ap,Di,SC,WZ}).

We remark here that the general form of the multiple Hurwitz zeta function $\zeta_k(s,a)$ was studied systematically
by Barnes \cite{Ba} during his research on the theory of the multiple Gamma function $\Gamma_k(a).$
In 1977, Shintani~\cite{Shintani} evaluated the special values of certain $L$-functions attached to real quadratic number fields in terms of  Barnes's multiple Hurwitz zeta functions,
and the work of Barnes has also been applied  in the study of the determinants of the Laplacians around the middle of 1980s (see \cite{Var}).

In this paper, by using the Jackson  integral, we extend Kurokawa, Mimachi and Wakayama's $q$-integral analogue  of the Raabe type formula for the Hurwitz zeta functions and log gamma functions to the multiple case. (See Theorems \ref{thm2} and \ref{thm-gamma} below).

Our paper is organized as follows. In Sec. 2, we shall recall Barnes and Vardi's definitions of multiple gamma functions. In Sec. 3 and Sec. 4, we shall state and prove Theorems \ref{thm2} and Theorem \ref{thm-gamma}, respectively.

\section{Multiple gamma functions}

The multiple gamma function $\Gamma_k(a)$ by the following recurrence-functional equation
(for references and a short historical survey see \cite{Ba,SC,Var}):
\begin{equation}\label{mul-gam-def}
\begin{aligned}
\Gamma_k(a+1)&=\frac{\Gamma_{k+1}(a)}{\Gamma_k(a)},\quad a\in\mathbb C,\; k\in\mathbb N, \\
\Gamma_1(a)&=\Gamma(a), \\
\Gamma_k(1)&=1.
\end{aligned}
\end{equation}
The multiple gamma functions were first defined in a slight different form by Barnes \cite{Ba}.
The definition (\ref{mul-gam-def}) is due to Vardi \cite{Var} and seems more natural than that of Barnes'.
Define
\begin{equation}\label{gam-def}
G_k(a)=\exp\left( {\zeta_k'(0,a)} \right),
\end{equation}
where
$$\zeta_k'(s,a)=\frac{\partial}{\partial s}\zeta_k(s,a).$$

\begin{proposition}[{Vardi \cite[Proposition 2.3]{Var}}] \label{m-gama-zet}
The multiple gamma function $\Gamma_k(a)$ may be expressed by
$$
\Gamma_k(a)=\left[\prod_{r=1}^k R_{k-r+1}^{(-1)^r\binom{a}{r-1}}\right] G_k(a),
$$
where $$R_m=\exp\left(\sum_{r=1}^m\zeta_r'(0,1)\right)\quad\text{with}\quad R_0=1.$$
\end{proposition}

In particular, the special cases of Proposition \ref{m-gama-zet} when $k=1$ and $k=2$ give other forms of
the simple and double gamma functions $\Gamma_1(a)=\Gamma(a)$ and $\Gamma_2(a)$, respectively:
$$\Gamma(a)=R_1^{-1}G_1(a)=\exp(-\zeta'(0)+\zeta'(0,a))=\sqrt{2\pi}\exp(\zeta'(0,a)),$$
where $\zeta(s)=\zeta(s,1)$ is the Riemann zeta function and
$$\begin{aligned}
\Gamma_2(a)&=R_2^{-1}R_1^aG_2(a) \\
&=\exp(-\zeta'(0)-\zeta_2'(0,1)+a\zeta'(0))\exp(\zeta_2'(0,a)) \\
&=\exp\left(\frac1{12}-\zeta'(-1)\right)\exp(-(1-a)\zeta'(0)) \exp\left(-\frac1{12}+\zeta_2'(0,a)\right)\\
&=A(2\pi)^{\frac12-\frac{a}2}\exp\left(-\frac1{12}+\zeta_2'(0,a)\right),
\end{aligned}$$
where we have used $\zeta_2(s,a)=(1-a)\zeta(s,a)+\zeta(s-1,a),\zeta'(0)=-\log\sqrt{2\pi}$ and
the known identity
$$\log A =\frac1{12}-\zeta'(-1),$$
in which $A$ is the Kinkelin's constant (see \cite[p.~496, p.~499 and Proposition 4.6]{Var}).

The Generalized Stieltjes constants are denoted as $\gamma_n(a)$ and arise when considering the Laurent expansion of the Hurwitz zeta
function $\zeta(s,a)$ at simple pole $s=1.$
For $0< a\leq1$ and for all complex $s\neq1,$ we have
\begin{equation}\label{st-la}
\zeta(s,a)=\frac1{s-1}+\sum_{n=0}^\infty\frac{(-1)^n}{n!}\gamma_n(a)(s-1)^n.
\end{equation}
(See \cite[Theorem 1]{Be}). The Euler constant $\gamma$ is a particular case of the Stieltjes constants, that is, $\gamma_0(1)=\gamma,$ and the Euler constant is expressed as
$$\gamma=\lim_{n\to\infty}\left(\sum_{i=1}^n\frac 1i-\log n\right)\approx 0.57721\ldots.$$

\section{Jackson's integral of multiple Hurwitz-Lerch zeta functions}

\subsection{Results} The multiple Lipschitz-Lerch zeta functions $\text{Li}_k(z,s)$ will be involved in our formula, which is  defined as
\begin{equation}\label{m-g-li}
\text{Li}_k(z,s)=\sum_{m_1,\ldots,m_k=1}^\infty\frac{z^{m_1+\dots+m_k}}{(m_1+\dots+m_k)^{s}},
\end{equation}
where $s\in\mathbb C$ when $|z|<1$ and $\Re(s)>k$ when $|z|=1.$
When $k=1$ in (\ref{m-g-li}), one knows that
\begin{equation}\label{pol-log}
\text{Li}_1(z,s)=\text{Li}_s(z)=\sum_{m=1}^\infty\frac{z^{m}}{m^{s}},
\end{equation}
which is the polylogarithm. Here $s$ is usually  called the order. For each fixed $s\in\mathbb C,$ the series (\ref{pol-log}) defines an
analytic function of $z$ for $|z|<1;$ in particular, $\text{Li}_0(z)=z/(1-z)$ and $\text{Li}_1(z)=-\log(1-z).$
Letting $z=1$ in (\ref{pol-log}), the Riemann zeta function $\text{Li}_s(1)=\zeta(s)$ is obtained.
It is easy to see that $\text{Li}_k(1,s)$ is expressed as the multiple Hurwitz zeta function
\begin{equation}\label{li-zeta}
\text{Li}_k(1,s)=\zeta_k(s,k).
\end{equation}
Denote by
$$[n]_{q}=\frac{q^{n}-1}{q-1}.$$
Our main result in this section is as follows:

\begin{theorem}\label{thm2}
Let $k$ be an positive integer.
Then
\begin{equation}\label{ra1}
\int_0^1 \Phi_k(z,s,a)d_qa=
\frac{1}{[1-s]_q}+\sum_{r=0}^{k-1}\binom kr \sum_{l=0}^\infty\binom{-s}{l}\frac{{\rm Li}_{k-r}(z,s+l)}{[l+1]_q}
\end{equation}
is established in the region $|z|\leq1$ and $\Re(s)<1.$
And
\begin{equation}\label{ra2}
\int_0^1 \Phi_k(z,s,k-a)d_qa=
z^{-k}\sum_{l=0}^\infty(-1)^l\binom{-s}{l}\frac{{\rm Li}_{k}(z,s+l)}{[l+1]_q},
\end{equation}
\begin{equation}\label{ra3}
\int_0^1 \Phi_k(z,s,k+a)d_qa=
z^{-k}\sum_{l=0}^\infty\binom{-s}{l}\frac{{\rm Li}_{k}(z,s+l)}{[l+1]_q}
\end{equation}
are established in the region
$s\in\mathbb C$ when $|z|<1$ and $s\in\mathbb C\setminus\{1,\ldots,k\}$ when $|z|=1.$
\end{theorem}

Letting $k=1$ and $z=1$ in Theorem \ref{thm2}, we immediately obtain the following result.

\begin{corollary}[{Kurokawa, Mimachi and Wakayama \cite[Theorem 2 and (3.1)]{KMW}}]\label{KMW}
We have
\begin{equation}\label{ra1w}
\int_0^1\zeta(s,a)d_qa=
\frac{1}{[1-s]_q}+\sum_{l=0}^\infty\binom{-s}{l}\frac{\zeta(s+l)}{[l+1]_q}
\end{equation}
for $\Re(s) < 1$, and
\begin{equation}\label{ra2w}
\int_0^1 \zeta(s,1-a)d_qa=
\sum_{l=0}^\infty(-1)^l\binom{-s}{l}\frac{\zeta(s+l)}{[l+1]_q},
\end{equation}
\begin{equation}\label{ra3w}
\int_0^1 \zeta(s,1+a)d_qa=
\sum_{l=0}^\infty\binom{-s}{l}\frac{\zeta(s+l)}{[l+1]_q}
\end{equation}
for all $s\in\mathbb{C}\setminus\{1\}.$
\end{corollary}

During the proof Theorem \ref{thm2}, we also obtain  the following result on the relationship between the higher and lower orders  Hurwitz zeta functions.

\begin{lemma}\label{thm1}
Let $\Phi_k(z,s,a)$ be the multiple Hurwitz-Lerch zeta function. Then
\begin{equation}
\Phi_k(z,s,a)=\frac1{a^s}
+\sum_{r=0}^{k-1}\binom{k}r z^{k-r}\Phi_{k-r}(z,s,a+(k-r))
\end{equation} is established in the region
$s\in\mathbb C$ when $|z|<1$ and $s\in\mathbb C\setminus\{1,\ldots,k\}$ when $|z|=1.$
\end{lemma}

 \subsection{Proof of Lemma \ref{thm1}}\label{pf-thm1}

First of all, letting  $k=2$ in (\ref{m-g-zeta}), we have
\begin{equation}\label{ex-2}
\begin{aligned}
\Phi_2(z,s,a)
&=\sum_{m_1,m_2=0}^\infty\frac{z^{m_1+m_2}}{(m_1+m_2+a)^{s}} \\
&=\sum_{m_2=1}^\infty\sum_{m_3=0}^\infty\frac{z^{m_2+m_3}}{(m_2+m_3+a)^{s}}
+\sum_{m_3=0}^\infty\frac{z^{m_3}}{(m_3+a)^{s}}.
\end{aligned}
\end{equation}
Thus
\begin{equation}\label{ex-3}
\begin{aligned}
\Phi_2(z,s,a)
&=\sum_{m_2,m_3=1}^\infty\frac{z^{m_2+m_3}}{(m_2+m_3+a)^{s}}
+\sum_{m_2=1}^\infty \frac{z^{m_2}}{(m_2+a)^{s}} \\
&\quad+\sum_{m_3=1}^\infty \frac{z^{m_3}}{(m_3+a)^{s}} +\frac1{a^s}.
\end{aligned}
\end{equation}
Letting $k=3$ in (\ref{m-g-zeta}), we have the following  decomposition:
\begin{equation}\label{ex-1}
\begin{aligned}
\Phi_3(z,s,a)
&=\sum_{m_1,m_2,m_3=0}^\infty\frac{z^{m_1+m_2+m_3}}{(m_1+m_2+m_3+a)^{s}} \\
&=\sum_{m_1=1}^\infty\sum_{m_2,m_3=0}^\infty\frac{z^{m_1+m_2+m_3}}{(m_1+m_2+m_3+a)^{s}}
+\Phi_2(z,s,a).
\end{aligned}
\end{equation}
We also have
\begin{equation}\label{sep-2}
\begin{aligned}
\sum_{m_1=1}^\infty\sum_{m_2,m_3=0}^\infty&\frac{z^{m_1+m_2+m_3}}{(m_1+m_2+m_3+a)^{s}}
=\sum_{m_1,m_2=1}^\infty\sum_{m_3=0}^\infty\frac{z^{m_1+m_2+m_3}}{(m_1+m_2+m_3+a)^{s}} \\
&\quad
+\sum_{m_1=1}^\infty\sum_{m_3=0}^\infty\frac{z^{m_1+m_3}}{(m_1+m_3+a)^{s}},
\end{aligned}
\end{equation}
\begin{equation}\label{sep-3}
\begin{aligned}
\sum_{m_1,m_2=1}^\infty\sum_{m_3=0}^\infty&\frac{z^{m_1+m_2+m_3}}{(m_1+m_2+m_3+a)^{s}}
=\sum_{m_1,m_2,m_3=1}^\infty\frac{z^{m_1+m_2+m_3}}{(m_1+m_2+m_3+a)^{s}} \\
&\quad
+\sum_{m_1,m_2=1}^\infty\frac{z^{m_1+m_2}}{(m_1+m_2+a)^{s}}
\end{aligned}
\end{equation}
and
\begin{equation}\label{sep-4}
\sum_{m_1=1}^\infty\sum_{m_3=0}^\infty\frac{z^{m_1+m_3}}{(m_1+m_3+a)^{s}}
=\sum_{m_1,m_3=1}^\infty\frac{z^{m_1+m_3}}{(m_1+m_3+a)^{s}}
+\sum_{m_1=1}^\infty\frac{z^{m_1}}{(m_1+a)^{s}}.
\end{equation}
In view of (\ref{ex-3}), (\ref{ex-1}), (\ref{sep-2}), (\ref{sep-3}) and (\ref{sep-4}), we get
\begin{equation}\label{ex-4}
\begin{aligned}
\Phi_3(z,s,a)
&=\frac1{a^s}+\binom32\sum_{m_2=1}^\infty \frac{z^{m_2}}{(m_2+a)^{s}} \\
&\quad
+\binom31\sum_{m_2,m_3=1}^\infty\frac{z^{m_2+m_3}}{(m_2+m_3+a)^{s}} \\
&\quad
+\binom30\sum_{m_1,m_2,m_3=1}^\infty\frac{z^{m_1+m_2+m_3}}{(m_1+m_2+m_3+a)^{s}}.
\end{aligned}
\end{equation}
In what follows, we use the above reasoning (the symmetry of the multiple sum
$\sum_{m_{1},\ldots,m_{k}=0}$ in the $m_{j}$'s variables) to prove an analogue of  (\ref{ex-4}) for general $k$.
Denote by
$$S=\{m_{1},\ldots,m_k\}.$$
Then
\begin{equation}\label{ex-5-0}
\begin{aligned}
\Phi_k(z,s,a)&=\sum_{m_1,\ldots,m_k=0}^\infty\frac{z^{m_1+\dots+m_k}}{(m_1+\dots+m_k+a)^{s}} \\
&=\frac1{a^s}+\sum_{r=0}^{k-1}\sum_{\{m_{i_1},\ldots,m_{i_{k-r}}\}\subset S}
\left[\sum_{m_{i_1},\ldots,m_{i_{k-r}}=1}^\infty
\frac{z^{m_{i_1}+\cdots+m_{i_{k-r}}}}{(m_{i_1}+\cdots+m_{i_{k-r}}+a)^{s}}\right].
\end{aligned}
\end{equation}
Since
$$\sum_{m_{i_1},\ldots,m_{i_{k-r}}=1}^\infty
\frac{z^{m_{i_1}+\cdots+m_{i_{k-r}}}}{(m_{i_1}+\cdots+m_{i_{k-r}}+a)^{s}}
=\sum_{m_{r+1},\ldots,m_k=1}^\infty
\frac{z^{m_{r+1}+\cdots+m_k}}{(m_{r+1}+\cdots+m_k+a)^{s}},$$
from (\ref{ex-5-0}), we have
\begin{equation}\label{ex-5}
\Phi_k(z,s,a)
=\frac1{a^s}
+\sum_{r=0}^{k-1}\binom{k}r\sum_{m_{r+1},\ldots,m_k=1}^\infty\frac{z^{m_{r+1}+\dots+m_k}}{(m_{r+1}+\dots+m_k+a)^{s}}.
\end{equation}
From (\ref{ex-5}), we get
\begin{equation}\label{ex-6}
\begin{aligned}
\Phi_k(z,s,a)
&=\frac1{a^s}+\sum_{r=0}^{k-1}\binom{k}r z^{k-r}\\
&\quad\times
\sum_{m_{r+1},\ldots,m_k=1}^\infty\frac{z^{(m_{r+1}-1)+\dots+(m_k-1)}}{((m_{r+1}-1)+\dots+(m_k-1)+a+(k-r))^{s}}\\
&=\frac1{a^s}+\sum_{r=0}^{k-1}\binom{k}r z^{k-r}
\Phi_{k-r}(z,s,a+(k-r))
\end{aligned}
\end{equation}
for $s\in\mathbb C$ when $|z|<1$ and $\Re(s)>k$ when $|z|=1,$ and analytically continued to $s\in\mathbb C\setminus\{1,\ldots,k\}$ if $|z|=1.$
This completes the proof.
\subsection{Proof of Theorem \ref{thm2}}\label{pf-thm2}

To prove Theorem \ref{thm2}, we need the following lemma.
\begin{lemma}\label{thm2-lem}
Let $k$ be an positive integer and $\Phi_k(z,s,a)$ be the multiple Hurwitz-Lerch zeta function.  Then
$$
\int_0^1\Phi_{k-r}(z,s,a+(k-r))d_qa
=z^{-(k-r)}\sum_{l=0}^\infty\binom{-s}{l}\frac{{\rm Li}_{k-r}(z,s+l)}{[l+1]_q}
$$
is established in the region $s\in\mathbb C$ when $|z|<1$ and $s\in\mathbb C\setminus\{1,\ldots,k\}$ when $|z|=1.$
\end{lemma}
\begin{proof}
Using (\ref{j-int}) and (\ref{m-g-li}), we obtain the following equality
$$
\begin{aligned}
\int_0^1\Phi_{k-r}&(z,s,a+(k-r))d_qa \\
&=(q-1)\sum_{n=1}^\infty\Phi_{k-r}(z,s,q^{-n}+(k-r))q^{-n} \\
&=\frac{q-1}{z^{k-r}}\sum_{n=1}^\infty\frac1{q^{n}}
\sum_{m_{r+1},\ldots,m_k=1}^\infty\frac{z^{m_{r+1}+\dots+m_k}}{(m_{r+1}+\dots+m_k+q^{-n})^{s}}
\\
&=\frac{q-1}{z^{k-r}}\sum_{n=1}^\infty\frac1{q^{n}}
\sum_{m_{r+1},\ldots,m_k=1}^\infty\frac{z^{m_{r+1}+\dots+m_k}}{(m_{r+1}+\dots+m_k)^{s}} \\
&\quad\times\left(1+\frac{ q^{-n}}{m_{r+1}+\dots+m_k}\right)^{-s}
\\
&=\frac{1}{z^{k-r}}\sum_{n=1}^\infty
\sum_{m_{r+1},\ldots,m_k=1}^\infty\frac{z^{m_{r+1}+\dots+m_k}}{(m_{r+1}+\dots+m_k)^{s}} \\
&\quad\times \sum_{l=0}^\infty\binom{-s}{l}\left(\frac{ 1}{m_{r+1}+\dots+m_k}\right)^l q^{-n(l+1)}(q-1) \\
&=\frac{1}{z^{k-r}}\sum_{l=0}^\infty\binom{-s}{l}\frac{{\rm Li}_{k-r}(z,s+l)}{[l+1]_q}
\end{aligned}
$$
for $s\in\mathbb C$ when $|z|<1$ and $\Re(s)>k$ when $|z|=1,$
and  analytically continued to $s\in\mathbb C\setminus\{1,\ldots,k\}$ if  $|z|=1.$
This completes the proof.
\end{proof}

\begin{proof}[\bf{Proof of Theorem \ref{thm2}.}] Now we  prove (\ref{ra1}). Suppose $|z|\leq1$ and $\Re(s)<1.$  From Lemma \ref{thm1}, we have
\begin{equation}\label{ra-eq1}
\begin{aligned}
\int_0^1\Phi_k(z,s,a)d_qa
=\int_0^1\left[\frac1{a^s}+\sum_{r=0}^{k-1}\binom{k}r z^{k-r}\Phi_{k-r}(z,s,a+(k-r))\right]d_qa.
\end{aligned}
\end{equation}
Thus from Lemma \ref{thm2-lem} and the following formula which is valid for $\Re(s)<1$ (e.g., see \cite[p.~50]{KMW}):
$$\int_0^1\frac1{a^s}d_qa
=(q-1)\sum_{n=1}^\infty\left(q^{-n}\right)^{-s}q^{-n}=\frac{1}{[1-s]_q},$$
we have shown (\ref{ra1}) in the region $\Re(s)<1.$

Now we prove (\ref{ra2}). Suppose $s\in\mathbb C$ when $|z|<1$ and $\Re(s)>k$ when $|z|=1.$
As in the prove of Lemma \ref{thm2-lem}, we have
$$
\begin{aligned}
\int_0^1\Phi_{k}&(z,s,k-a)d_qa \\
&=(q-1)\sum_{n=1}^\infty\Phi_{k}(z,s,k-q^{-n})q^{-n} \\
&=(q-1)\sum_{n=1}^\infty\frac1{q^{n}}
\sum_{m_1,\ldots,m_k=0}^\infty\frac{z^{m_{1}+\dots+m_k}}{(m_{1}+\dots+m_k+k-q^{-n})^{s}}  \\
&=\frac{q-1}{z^{k}}\sum_{n=1}^\infty\frac1{q^{n}}
\sum_{m_1,\ldots,m_k=1}^\infty\frac{z^{m_{1}+\dots+m_k}}{(m_{1}+\dots+m_k-q^{-n})^{s}}  \\
&=\frac{q-1}{z^{k}}\sum_{n=1}^\infty\frac1{q^{n}}
\sum_{m_1,\ldots,m_k=1}^\infty\frac{z^{m_{1}+\dots+m_k}}{(m_{1}+\dots+m_k)^{s}} \\
&\quad\times\left(1-\frac{ q^{-n}}{m_{1}+\dots+m_k}\right)^{-s} 
\\
&=\frac1{z^{k}}\sum_{n=1}^\infty
\sum_{m_1,\ldots,m_k=1}^\infty\frac{z^{m_{1}+\dots+m_k}}{(m_{1}+\dots+m_k)^{s}} \\
&\quad\times \sum_{l=0}^\infty(-1)^l\binom{-s}{l}\left(\frac{1}{m_{1}+\dots+m_k}\right)^l q^{-n(l+1)}(q-1) \\
&=\frac1{z^{k}}\sum_{l=0}^\infty(-1)^l\binom{-s}{l}\frac{{\rm Li}_{k}(z,s+l)}{[l+1]_q},
\end{aligned}
$$
so it is valid for $s\in\mathbb C$ when $|z|<1$, and from the analytic continuation it is also valid for all $s\in\mathbb C\backslash\{1,\ldots,k\}$ when $|z|=1$.
This completes the proof of (\ref{ra2}).
The same procedure proves (\ref{ra3}).
\end{proof}

\section{Jackson's integral of multiple gamma functions}
\subsection{Results}
Let $\left[  \begin{smallmatrix}  r \\ i  \end{smallmatrix} \right]$
be the absolute Stirling numbers of the first kind (see, e.g., \cite[Section 1.6]{SC}), defined recursively by
\begin{equation}\label{frist-st}
\left[  \begin{array}{c}  r \\ i  \end{array} \right]=(r-1)\left[  \begin{array}{c}  r-1 \\ i  \end{array} \right]
+\left[  \begin{array}{c}  r-1 \\ i-1  \end{array} \right], \quad
\left[  \begin{array}{c}  r \\ 0  \end{array} \right]=\begin{cases}
1 &\text{if } r=0, \\
0 &\text{if } r\neq0,
\end{cases}
\end{equation}
(see \cite[(12)]{Ad}).

Let $H_{l}:=\sum_{j=1}^{l} \frac1j~(l\in\mathbb{N})$ be the harmonic numbers (see, e.g., \cite[p. 77]{SC}).
By virtue of Theorem \ref{thm2} with $z=1$ and Proposition \ref{m-gama-zet} above, we obtain the following result.

\begin{theorem}\label{thm-gamma}
Let $k$ be an positive integer. Then we have
$$
\begin{aligned}
\int_0^1\log & \Gamma_k(a)d_qa =\sum_{r=1}^k\frac{(-1)^r\log R_{k-r+1}}{(r-1)!}\sum_{i=0}^{r-1}(-1)^{r-i-1}\left[  \begin{array}{c}  r-1 \\ i  \end{array} \right]\frac1{[i+1]_q} \\
&\quad+\frac{q\log q}{q-1}+\sum_{r=0}^{k-1}\binom kr\zeta'_{k-r}(0,k-r) \\
&\quad+\sum_{r=0}^{k-1}\binom kr \sum_{l=1}^{k-r}\frac{(-1)^l}{l[l+1]_q}\biggl[
\sum_{\substack{j=0 \\ j\neq l-1}}^{k-r-1}\frac{P_{j,k-r}(k-r)}{(k-r-1)!}\zeta(l-j,k-r) \\
&\qquad\qquad + \frac{P_{l-1,k-r}(k-r)}{(k-r-1)!}\left(H_{l-1}+\gamma_0(k-r)\right) \biggl] \\
&\quad+\sum_{r=0}^{k-1}\binom kr \sum_{l=k-r+1}^{\infty}\frac{(-1)^l}{l[l+1]_q}\zeta_{k-r}(l,k-r),
\end{aligned}
$$
where
$$P_{j,k}(a)=\sum_{i=j+1}^k(-a)^{i-j-1}\binom{i-1}{j}\left[  \begin{array}{c}  k \\ i  \end{array} \right]$$
and $\left[  \begin{smallmatrix}  k \\ i  \end{smallmatrix} \right]$ is the absolute Stirling numbers of the first kind.
Here, the sum $H_{0}$ is understood to be $0$.
\end{theorem}

The case $k=1$ of Theorem \ref{thm-gamma} reduces to

\begin{corollary}[{Kurokawa, Mimachi and Wakayama \cite[Theorem 3]{KMW}}]\label{cor-thm3}
$$\int_0^1\log \Gamma(a)d_qa = \frac{q\log q}{q-1}-\frac{\gamma}{[2]_q}+\sum_{l=2}^\infty\frac{(-1)^l\zeta(l)}{l[l+1]_q}.$$
\end{corollary}

\subsection{Proof of Theorem \ref{thm-gamma}}\label{thm-ga-diff}

The proof of Theorem \ref{thm-gamma} is based on the following two lemmas.

\begin{lemma}\label{logGk}
We have
$$\begin{aligned}
\int_0^1&\log G_k(a)d_qa =\int_0^1\log\Gamma_k(a)d_qa  \\
&\quad-\sum_{r=1}^k\frac{(-1)^r\log R_{k-r+1}}{(r-1)!}\sum_{i=0}^{r-1}(-1)^{r-i-1}\left[  \begin{array}{c}  r-1 \\ i  \end{array} \right]\frac1{[i+1]_q}.
\end{aligned}$$
\end{lemma}
\begin{proof}
From (\ref{gam-def}) and Proposition \ref{m-gama-zet}, the multiple gamma function $\Gamma_k(a)$ may be expressed by
\begin{equation}
\log\Gamma_k(a)=\log G_k(a)+\sum_{r=1}^k(-1)^r\binom a{r-1}\log R_{k-r+1},
\end{equation}
which is equivalent to
\begin{equation}\label{log-Gk}
\int_0^1\log G_k(a)d_qa =\int_0^1\log\Gamma_k(a)d_qa-\sum_{r=1}^k(-1)^r\frac{\log R_{k-r+1}}{(r-1)!}\int_0^1(a)_{r-1}d_qa,
\end{equation}
where $(a)_r$ denotes the falling factorial power.
Note that
\begin{equation}\label{int-q}
\int_0^1a^id_qa=\frac1{[i+1]_q}
\end{equation}
for $i\geq0$ (see \cite[p.~47]{KMW}) and
\begin{equation}\label{frist-st-pro}
(a)_r=a(a-1)\cdots(a-r+1)=\sum_{i=0}^r (-1)^{r-i}\left[  \begin{array}{c}  r \\ i  \end{array} \right] a^i,
\end{equation}
where $\left[  \begin{smallmatrix}  k \\ i  \end{smallmatrix} \right]$ is  the absolute Stirling numbers of the first kind.
Using (\ref{log-Gk}), (\ref{int-q}) and (\ref{frist-st-pro}), the desired formula follows.
\end{proof}

\begin{lemma}[{Choi \cite[Eq. (2.5)]{Choi}}]\label{mul-zeta-ad}
The multiple zeta function $\zeta_k(s,a)$ defined by (\ref{m-zeta_z=1}) may be expressed by means of the Hurwitz function
$$\zeta_k(s,a)=\frac1{(k-1)!}\sum_{j=0}^{k-1}P_{j,k}(a)\zeta(s-j,a).$$
\end{lemma}

Now we are at the position to prove our main result in this section.
\begin{proof}[\bf{Proof of Theorem \ref{thm-gamma}.}]
From Theorem \ref{thm2} with $z=1,$ we obtain
\begin{equation}\label{zeta-ra1}
\int_0^1 \zeta_k(s,a)d_qa=
\frac{1}{[1-s]_q}+\sum_{r=0}^{k-1}\binom kr \sum_{l=0}^\infty\binom{-s}{l}\frac{\zeta_{k-r}(s+l,k-r)}{[l+1]_q},
\end{equation}
where we have used ${\rm Li}_k(1,s)=\zeta_k(s,k).$

The following identities  will be used to prove the theorem:

\begin{enumerate}
\item $\lim_{s\to0}\left(\frac{\partial \zeta_k(s,a)}{\partial s}\right)=\log G_k(a).$
\item $\lim_{s\to0}\left(\frac{\partial}{\partial s}\frac1{[1-s]_q}\right)=\frac{q\log q}{q-1}.$
\item $\lim_{s\to0}\left(\frac{\partial}{\partial s}\binom{-s}{l}\right)=\lim_{s\to0}\left(\frac{\partial}{\partial s}
\frac{(-s)(-s-1)\cdots(-s-l+1)}{l!}\right)=\frac{(-1)^l}{l}.$
\item From (\ref{st-la}), we have
$$\begin{aligned}\lim_{s\to0}\frac{\partial}{\partial s}\left( s\zeta(s+1,a)\right)
&=\lim_{s\to0}\frac{\partial}{\partial s}\left(1+\sum_{n=0}^\infty\frac{(-1)^n}{n!}\gamma_n(a)s^{n+1}\right) \\
&=\gamma_0(a).
\end{aligned}$$
\item For $1\leq r\leq k,$ we have
$$\lim_{s\to0}\frac{\partial}{\partial s}\left(\binom{-s}{0}\zeta_{r}(s,r)\right)=\zeta'_{r}(0,r).$$
\item Let $1\leq r\leq k.$ By Lemma \ref{mul-zeta-ad} and   (4) above, we have
$$\begin{aligned}
\lim_{s\to0}\frac{\partial}{\partial s}&\left(\binom{-s}{1}\zeta_{r}(s+1,r)\right) \\
&=-\frac1{(r-1)!}\sum_{j=0}^{r-1}P_{j,r}(r)\lim_{s\to0}\frac{\partial}{\partial s}\left(s\zeta(s+1-j,r)\right) \\
&=-\frac1{(r-1)!}\left(P_{0,r}(r)\gamma_0(r)+ \sum_{j=1}^{r-1}P_{j,r}(r)\zeta(1-j,r)\right).
\end{aligned}$$
\item Let $1\leq r\leq k.$ By Lemma \ref{mul-zeta-ad} and the above (4), we have
$$\begin{aligned}
\lim_{s\to0}\frac{\partial}{\partial s}&\left(\binom{-s}{i}\zeta_{r}(s+i,r)\right) \\
&=\frac1{(r-1)!}\sum_{j=0}^{r-1}P_{j,r}(r)\lim_{s\to0}\frac{\partial}{\partial s}\left(\binom{-s}i\zeta(s+i-j,r)\right) \\
&=\frac1{(r-1)!}\frac{(-1)^i}{i}\sum_{\substack{j=0 \\ j\neq i-1}}^{r-1}P_{j,r}(r)\zeta(i-j,r)  \\
&\quad+\frac1{(r-1)!}P_{i-1,r}(r)\frac{(-1)^i}{i}\left(H_{i-1}+\gamma_0(r)\right),
\end{aligned}$$
where we have used
$$\begin{aligned}
\lim_{s\to0}\frac{\partial}{\partial s}&\left(\binom{-s}i\zeta(s+1,r)\right) \\
&=\frac{(-1)^i}{i}\lim_{s\to0}\frac{\partial}{\partial s}\left[
(s+1)\cdots(s+i-1)
\left(1+\sum_{n=0}^\infty\frac{(-1)^n}{n!}\gamma_n(r)s^{n+1}\right)\right] \\
&=\frac{(-1)^i}{i!}\left(H_{i-1}+\gamma_0(r)\right).
\end{aligned}$$
\end{enumerate}

Differentiating both sides of (\ref{zeta-ra1}) and then setting $s=0,$ we find that
\begin{equation}\label{diff-zeta-1}
\begin{aligned}
\int_0^1 \lim_{s\to0}\frac{\partial}{\partial s}&\zeta_k(s,a)d_qa
=\lim_{s\to0}\frac{\partial}{\partial s}\left(\frac{1}{[1-s]_q}\right) \\
&+\sum_{r=0}^{k-1}\binom kr \sum_{l=0}^\infty \frac1{[l+1]_q}
\lim_{s\to0}\frac{\partial}{\partial s}\left[\binom{-s}{l}\zeta_{k-r}(s+l,k-r)\right],
\end{aligned}
\end{equation}
then after using (1)-(7) and replacing $r$ by $k-r,$ we obtain
$$
\begin{aligned}
\int_0^1&\log G_k(a)d_qa =
\frac{q\log q}{q-1}+\sum_{r=0}^{k-1}\binom kr\zeta'_{k-r}(0,k-r) \\
&\quad+\sum_{r=0}^{k-1}\binom kr \sum_{l=1}^{k-r}\frac{(-1)^l}{l[l+1]_q}\biggl[
\sum_{\substack{j=0 \\ j\neq l-1}}^{k-r-1}\frac{P_{j,k-r}(k-r)}{(k-r-1)!}\zeta(l-j,k-r) \\
&\qquad\qquad + \frac{P_{l-1,k-r}(k-r)}{(k-r-1)!}\left(H_{l-1}+\gamma_0(k-r)\right) \biggl] \\
&\quad+\sum_{r=0}^{k-1}\binom kr \sum_{l=k-r+1}^{\infty}\frac{(-1)^l}{l[l+1]_q}\zeta_{k-r}(l,k-r),
\end{aligned}
$$
where the sum $H_{0}$ is understood to be $0$.
Therefore, by Lemma \ref{logGk}, we obtain the desired result.
\end{proof}

\section*{Acknowledgment} The authors are grateful to the anonymous referee for his/her helpful comments.

\bibliography{central}

\begin{thebibliography}{00}


\bibitem{Ad} V.S. Adamchik,
\textit{The multiple gamma function and its application to computation of series},
Ramanujan J. \textbf{9} (2005), no. 3, 271--288.

\bibitem{Ap} T.M. Apostol, 
\textit{On the Lerch zeta function}, 
Pac. J. Math. \textbf{1} (1951), 161--167.

\bibitem{Ba} E.W. Barnes, 
\textit{On the theory of the multiple gamma functions}, 
Trans. Cambridge Philos. Soc. \textbf{19} (1904), 374--425.

\bibitem{Be} B.C. Berndt, 
\textit{On the Hurwitz zeta-function}, 
Rocky Mountain J. Math. \textbf{2} (1972), no. 1, 151--157.

\bibitem{Choi} J. Choi, 
\textit{Explicit formulas for Bernoulli polynomials of order $n$}, 
Indian J. Pure Appl. Math. \textbf{27} (1996), no. 7, 667--674.

\bibitem{CJS}J. Choi, D.S. Jang and H.M. Srivastava,
\textit{A generalization of the Hurwitz-Lerch zeta function}, 
Integral Transforms Spec. Funct. \textbf{19} (2008), no. 1-2, 65--79.


\bibitem{Di} K. Dilcher, 
\textit{Sums of products of Bernoulli numbers}, 
J. Number Theory \textbf{60} (1996), 23--41.

\bibitem{Jack} F.H. Jackson, 
\textit{On $q$-definite integrals},
Quart. J. Pure Appl. Math. \textbf{41} (1910), no. 15, 193--203.


\bibitem{KMW} N. Kurokawa, K. Mimachi and M. Wakayama,
\textit{Jackson's integral of the Hurwitz zeta function}, 
Rend. Circ. Mat. Palermo (2) \textbf{56} (2007), no. 1, 43--56.

\bibitem{Lerch} M. Lerch, 
\textit{Dals\v si studie v oboru Malmst\'enovsk\'ych  \v rad}, 
Rozpravy \v Cesk\'e Akad., \textbf{3} (1894), no. 28, 1--61.

\bibitem{Nielsen} N. Nielsen,
\textit{Handbuch der Theorie der Gammafunktion}, 
Chelsea, New York, 1965, reprint of 1906 edition.

\bibitem{Shintani} T. Shintani,
\textit{On a Kronecker limit formula for real quadratic fields}, 
J. Fac. Sci. Univ. Tokyo Sect. IA Math. \textbf{24} (1977), no. 1, 167--199.

\bibitem{SC} H.M. Srivastava and J. Choi,
\textit{Zeta and $q$-Zeta Functions and Associated Series and Integrals}, 
Elsevier Science Publishers, Amsterdam, London and New York, 2012.

\bibitem{Var} I. Vardi, 
\textit{Determinants of Laplacians and multiple Gamma functions}, 
SIAM J. Math. Anal. \textbf{19} (1988), 493-–507.

\bibitem{WZ} K.S. Williams and N.Y. Zhang,
\textit{Special values of the Lerch zeta function and the evaluation of certain integrals},
Proc. Amer. Math. Soc. \textbf{119} (1993), 35--49.

\end{thebibliography}

\end{document}